\documentclass[11pt]{amsart}
\usepackage{amsmath, amsrefs, amsthm, amssymb}
\usepackage{pdfsync}					
\newcommand{\dd}{\,{\rm d}}

\newcommand\R{{\mathbb{R}}}

\newcommand\Z{{\mathbb{Z}}}

\renewcommand\S{{\mathbb{S}}}

\newtheorem{theorem}{Theorem}[section]

\newtheorem{lemma}[theorem]{Lemma}
\newtheorem{corollary}[theorem]{Corollary}

\theoremstyle{definition}

\theoremstyle{remark}
\newtheorem{remark}[theorem]{Remark}

\numberwithin{equation}{section}

\def\R{\mathbb{R}}

\begin{document}

\title[Generalized hyperelastic rod equation]
{Blowup issues for a class of nonlinear dispersive wave equations}

\author{Lorenzo Brandolese and Manuel Fernando Cortez}
\thanks{The authors are supported by the French ANR Project DYFICOLTI. The second author is also supported by the 
{\it Secretar\'\i a Nacional de Educaci\'on Superior, Ciencia, Tecnolog\'\i a e Innovaci\'on\/}}

\address{Universit\'e de Lyon, Universit\'e Lyon 1,
CNRS UMR 5208 Institut Camille Jordan,
43 bd. du 11 novembre,
Villeurbanne Cedex F-69622, France.}
\email{brandolese{@}math.univ-lyon1.fr, cortez@math.univ-lyon1.fr}
\urladdr{http://math.univ-lyon1.fr/$\sim$brandolese}


\begin{abstract}
In this paper we consider  the nonlinear dispersive wave equation  on the real line,
$u_t-u_{txx}+[f(u)]_x-[f(u)]_{xxx}+\bigl[g(u)+\frac{f''(u)}{2}u_x^2\bigr]_x=0$,
that for appropriate choices of the functions~$f$ and~$g$ includes well known models,
such as Dai's equation for the study of vibrations inside elastic rods or the Camassa--Holm equation
 modelling water wave propagation in shallow water.
 We establish a \emph{local-in-space}  blowup criterion (\emph{i.e.}, a criterion 
 involving  only the properties of the data~$u_0$ in a neighbourhood of 
 a~\emph{single point}) simplifying and extending earlier blowup criteria
 for this equation. Our arguments apply both to the finite and infinite energy case,  yielding the finite time blowup
 of strong solutions with possibly different behavior as $x\to+\infty$ and $x\to-\infty$.
\end{abstract}

 \maketitle

\section{Introduction}

The experimental observation by the naval architect Scott Russel
of solitary waves propagating in channels at different speeds, and 
interacting in a nonlinear way before recovering  their initial shape,
motivated the studies on the mathematical modelling of  water wave motion at the end of the XIX  century.
The first works can be retraced back to Boussinesq, Rayleigh, Korteweg and  de Vries.
The celebrated KdV equation allows for a first mathematical description of such phenomena.
This equation can be derived as an asymptotic model from the free surface Euler equations in the so called shallow water regime 
$\mu=h^2/\lambda^2<\!\!\!<1$, where $h$ and $\lambda$ denote respectively the average elevation of the liquid over the bottom and the characteristic wavelengt. It models small amplitude waves, {\it i.e.\/} waves such that the dimensionless amplitude
parameter  $\epsilon=a/h$ satisfies $\epsilon=O(\mu)$, where $a$ is the typical amplitude.

Such small amplitude waves feature both nonlinear and dispersive effects. For larger amplitude waves
nonlinear effects become preponderant and wave breaking can eventually occur.
As the KdV equation is no longer suitable for the description of breaking mechanisms --- its solutions remain smooth for all time ---
a considerable effort was made toward the modelling of larger amplitude, possibly breaking waves, see, {\it e.g.\/}, the monograph~\cite{Whit99}.
Such studies culminated with the derivation in 1993, by Camassa and Holm,~\cites{CamHol93, CamHolHym94}
of an equation  obtained from the vertically averaged water wave system, written in Lie-Poisson Hamiltonian form, by
an asymptotic expansion preserving the Hamiltonian structure.
The scaling of validity of the Camassa--Holm equation is $\mu<\!\!\!<1$ and $\epsilon=O(\sqrt \mu)$: such scaling includes 
that of KdV allowing higher order accuracy.
Alternative rigorous derivations of the Camassa--Holm equation are also available, see~\cites{AConL09, John02}.
Such equation attracted a considerable interest in the past 20 years, not only due its hydrodynamical relevance (it was the first
equation capturing both soliton-type solitary waves as well as breaking waves) but also because of its extremely rich
mathematical structure.
In fact,  the Camassa--Holm equation was written for the first time in a different context, as one of the 12 integrable equations classified by Fokas and Fuchsteiner~\cite{FuchFok81} and obtained from a nonlinear operator satisfying suitable defining properties, applying a recursive operator that is an hereditary symmetry.

The Camassa--Holm equation is usually written as
\begin{equation}
 \label{CH}
 u_t+\kappa u_x-u_{xxt}+3uu_x=uu_{xxx}+2u_xu_{xx}, \qquad x\in\R,\quad t>0.
\end{equation}
where $u$ can be interpreted as an horizontal velocity of the water at a certain depth and $\kappa$ is the dispersion parameter.

The dispertionless case~$\kappa=0$ is of mathematical interest as in this case the equation possess soliton solutions
(often named {\it peakons\/}) peaked at their crest, of the form $u_c(t,x)=ce^{-|x-ct|}$. Multi-peakon interactions is studied in 
\cites{CamHol93, CamHolHym94}.
For $\kappa>0$ the equation admits smooth solitons.

In the shallow water interpretation,  however, $\kappa$ is proportional to the square root of the water depth and cannot be zero.
On the other hand, the same equation, with~$\kappa=0$ appears, 
{\it e.g.\/}, in the study of the dynamics of a class of non-Newtonian, second-grade fluids (see \cite{Bus99}), or
when modelling vibrations inside hyper-elastic rods.
In the latter case peakons correspond to physical solutions.
More in general,  the propagation of nonlinear waves inside cylindrical hyper-elastic rods, assuming that the diameter is small when compared to the axial length scale, is described by the one dimensional equation (see \cite{Dai98}),
\[
v_\tau+\sigma_1 v v_\xi+\sigma_2 v_{\xi\xi\tau}+\sigma_3(2v_\xi v_{\xi\xi}+vv_{\xi\xi\xi})=0,
\qquad
\xi\in\R, \; \tau>0.
\]
Here $v(\tau,\xi)$ represents the radial stretch relative to a pre-stressed state, $\sigma_1\not=0$, $\sigma_2<0$ and
$\sigma_3\le0$ are physical constants depending on the material.
The scaling transformations
\[
\tau=\frac{3\sqrt{-\sigma_2}}{\sigma_1}t,\qquad
\xi=\sqrt{-\sigma_2}x,
\]
with $\gamma=3\sigma_3/(\sigma_1\sigma_2)$ and $u(t,x)=v(\tau,\xi)$,
allow us to reduce the above equation to
\begin{equation}
\label{rod-pde}
u_t-u_{xxt}+3uu_x=\gamma(2u_xu_{xx}+uu_{xxx}),
\qquad x\in\R,\;t>0.
\end{equation}
Notice that when the physical parameter $\gamma$ (related to the Finger deformation tensor) is equal to~$1$, one recovers the
dispersionless Camassa--Holm equation. Several positive or
negative values of~$\gamma$ correspond to known hyper-elastic materials.

Common important features of the Camassa--Holm and the rod equation include:
\begin{itemize}
 \item[-] The conservation of the energy integral $\int (u^2+u_x^2)$ for classical and sufficiently decaying solutions.
 \item[-] The local well-posedness theory: the Cauchy problems for equation~\eqref{CH} and~\eqref{rod-pde} are well-posed in
 the Sobolev space $H^s(\R)$ for $s>3/2$ (or in suitable Besov spaces), locally in time.
See, {\it e.g.}, \cites{ConStra00, Dan01, LiOlver00, Yin04}.
 \item[-] Wave breaking scenario: the maximal existence time $T^*$ such that the solution belongs to 
 $C([0,T^*),H^s(\R))\cap C^1([0,T^*),H^{s-1}(\R))$
 is finite if and only if $u_x$ is unbounded from below (or from above when $\gamma<0$) near the blowup time. Up to the time $T^*$,
 the solution remains uniformly bounded. See \cites{CamHol93, ConStra00, Yin04}. 
 \item[-] Finite time blowup criteria on the initial data and upper bound estimates on $T^*$.
 See \cites{ConEschActa, ConStra00, McKean04, Liu-MathAnn06, Yin04, Zhou2004, Zhou-Math-Nachr}.
 \item[-] Exponentially decaying solitary wave solutions.
 See, {\it e.g.\/} \cites{CamHolHym94, Dai98, Dai-Huo,  Lenells-DCDS06, Yin04}.
 \item[-] Infinite propagation speed and persistence results in weighted spaces. See \cites{Bra11, HMPZ07}.
 \item[-] The existence of global conservative or dissipative weak solutions. 
 See \cites{BreCon07, BreCon07bis, HolRay07, HolJDE07bis}.
\end{itemize}

Beside such common features, the Camassa--Holm equation is considerably better understood than the rod equation.
Indeed, contrary to the rod equation with $\gamma\not=1$, equation~\eqref{CH} 
possess a bi-hamiltonian structure that makes the equation formally integrable via the inverse scattering method.
Elegant geometric interpretations (see \cites{ConKol03, ConMc, Kol07, Misio98}) are available: for example, 
equation~\eqref{CH} gives rise to a geodesic flow of a right invariant metric on the Bott-Virasoro group.
This equation also admits infinitely many conservation laws. Moreover, it has solitary waves interacting 
like solitons that are orbitally stable, see~\cite{AConS00}.
 The global existence of strong solutions of the Camassa--Holm equation can be obtained putting suitable sign conditions on the associated initial potential $y_0=u_0-(u_0)_{xx}$, see \cites{ACon00, McKean04}. When $\gamma\not=1$, as the sign of the potential
is no longer conserved by the flow, such global existence criterion is no longer valid. In fact, we know of no general condition on the initial datum guaranteeing that the corresponding solution of the rod equation remains in~$H^s$ ($s>3/2$) for all time.
The value of the parameter~$\gamma$ plays a crucial role: two limit situations occur when~$\gamma=0$ and $\gamma=3$.
In the latter case {\it any\/} non-zero solution  eventually develops a singularity (see \cite{ConStra00}). 
On the other hand, when~$\gamma=0$, no blowup can occur. (When $\gamma=0$, the rod 
equation boils down to the BBM equation, a model 
for the unidirectional evolution of long waves, \cite{BBM72}).
Physically, the formation of a singularity corresponds to a formation of a crack inside the rod.
Both behaviors in the two limit cases $\gamma=0$ and $\gamma=3$ are not physically realistic for real rods, but 
it is interesting to observe that there are known materials such that~$\gamma$ is indeed close to~$3$ ({\it e.g.\/}, $\gamma=2.668$)
and others such that~$\gamma$ is close to zero ({\it e.g.\/} $\gamma=-0.539$).
We refer to~\cite{Dai-Huo} for a list of physically acceptable values of~$\gamma$ (ranging from~$-29.476$ to $3.417$).

 \section{Blowup for the generalized hyper-elastic rod equation}

In this paper we will consider the Cauchy problem for a nonlinear dispersive wave equation including
both~\eqref{CH} and \eqref{rod-pde} as a particular case: 
\begin{equation}
 \label{gr-pde}
 u_t-u_{txx}+[f(u)]_x-[f(u)]_{xxx}+\Bigl[g(u)+\frac{f''(u)}{2}u_x^2\Bigr]_x=0.
\end{equation}
Equation~\eqref{gr-pde} is often referred as the {\it generalized hyper-elastic rod wave equation}, see~\cite{HolRay07}.
The Camassa--Holm equation corresponds to the choice $f(u)=u^2/2$,  $g(u)=\kappa u+u^2$ and the rod equation to the choice
$f(u)=\frac{\gamma}{2}u^2$ and $g(u)=\frac{3-\gamma}{2}u^2$.
When $f(u)=\frac{u^{Q+1}}{Q+1}$ and $g(u)=\kappa u+\frac{Q^2+3Q}{2(Q+1)}u^{Q+1}$ one recovers from~\eqref{gr-pde}
another class of equations with interesting mathematical properties, studied in~\cite{HakKir}.

From now on, we will study the Cauchy problem for the generalized rod equation, written in the non-local form, formally equivalent 
to~\eqref{gr-pde}:
 \begin{equation}
 \label{genrod}
  \begin{cases}
   u_t+f'(u)u_x+\partial_xp*\Bigl[g(u)+\textstyle\frac{f''(u)}{2}u_x^2\Bigr]=0, &x\in\R,\quad t>0,\\
   u(x,0)=u_0(x), & x\in\R.
  \end{cases}
\end{equation}
 Here
 \[
  p(x)=\textstyle\frac12e^{-|x|}
 \]
 is the fundamental solution of the operator $1-\partial^2_x$.
 The problem~\eqref{genrod} is thus written in the abstract form
 \begin{equation*}
  \frac{\dd u}{\dd t} +A(u)= H(u), \qquad u(x,0)=u_0(x),
 \end{equation*}
with $A(u)=f'(u)\partial_x$ and $H(u)=-\partial_x(1-\partial_x^2)^{-1}\bigl[g(u)+\frac{f''(u)}{2}u_x^2\bigr]$.
The local existence theory can be developped applying classical Kato's approach~\cite{Kat75}. 
For reader's convenience we collect in a single theorem the main results of the recent paper of Tian, Yan and Zhang \cite{TYZ14}
on the problem~\eqref{genrod}.

 \begin{theorem}[See \cite{TYZ14}]
 \mbox{}
  \label{th:TYZ}
  \begin{enumerate}
 \item
 Assume that $f,g\in C^\infty(\R)$. Let  $u_0\in H^s(\R)$, $s>3/2$. 
 Then there exists $T>0$, with $T=T(u_0,f,g)$ and a unique solution $u$ to the Cauchy problem~\eqref{genrod} such that 
 $u\in C([0,T),H^s(\R))\cap C^{1}([0,T),H^{s-1}(\R))$.  The solution has constant energy integral
 \begin{equation*}
\label{en-inv}
 \int_{\R} (u^2+(u_x)^2)=\int_{\R} (u_0^2+(u_{0}')^2)=\|u_0\|_{H^1}^2.
 \end{equation*}
 Moreover, the solution depends continuously on the initial data: the mapping $u_0\mapsto u$ is continuous from $H^s(\R)$ to $C([0,T),H^s(\R))\cap C^{1}([0,T),H^{s-1}(\R))$.
 \item Assume in addition that $f''\ge\gamma>0$:
 \begin{itemize}
 \item [i)](Blowup scenario and rate)
 Let $0<T^*\le\infty$ be the maximal time of the solution in~$C([0,T^*),H^s(\R))\cap C^1([0,T^*),H^{s-1}(\R))$.
 Then $T^*<\infty$ if and only if 
 \begin{equation*}
 \label{scenario}
 \lim_{t\to T^*} \inf_{x\in\R} u_x(t,x)=-\infty.
 \end{equation*}
 In this case, the blowup rate of $\inf_{x\in\R} u_x(t,x)$ as $t\to T^*$ is $O(\frac{1}{T^*-t})$.
 \item[ii)](Blowup criterion)
 Assume that there exists a point $x_0\in\R$ such that
\begin{equation}
 \label{blow:TYZ}
 u_0'(x_0)<-\sqrt\frac{4 \sup_{|v|\le \|u_0\|_{H^1}} |g(v)|+{\|u_0\|_{H^1}^2}\sup_{|v|\le \|u_0\|_{H^1}}f''(v)}{\gamma}.
\end{equation}
Then~$u$ blows up in finite time and 
$T^*\le \frac{1}{\sqrt{2C_0\gamma}}\log\left(\frac{\sqrt{\gamma/2}\, u_0'(x_0)-\sqrt C_0}{\sqrt{\gamma/2}\, u_0'(x_0)+\sqrt C_0}\right)$,
where $C_0=C_0(\|u_0\|_{H^1},f,g)$ is given by 
 \begin{equation*}
  \label{def:K}
  C_0\equiv 2 \sup_{|v|\le \|u_0\|_{H^1}} |g(v)|+\frac{\|u_0\|_{H^1}^2}{2}\sup_{|v|\le \|u_0\|_{H^1}}f''(v).
 \end{equation*}
\end{itemize} 
\end{enumerate}
 \end{theorem}
 In the first item, the existence time~$T$ can be taken independent on the parameter~$s$ in the following sense: 
 if $u_0$ also belongs to $H^{s_1}(\R)$  with $s_1>3/2$, then we have also 
 $u\in C([0,T),H^{s_1}(\R))\cap C^{1}([0,T),H^{s_1-1}(\R))$. 
Additional results in~\cite{TYZ14} include lower bound estimates for the existence time~$T^*$ and the lower semi-continuity of the
existence time. Let us also mention the construction of global conservative weak solutions for such equation
\cite{HolRay07} (see also \cites{BreCon07, BreCon07bis, CHK05} for earlier results on weak solutions for more
specific choices of the functions~$f$ and $g$).

\medskip
 \subsection*{Main results}
The purpose of this paper is to establish a new blowup criterion for equation~\eqref{genrod}, considerably simplifying~\eqref{blow:TYZ}, and 
extending our previous result established in~\cite{Bra13} in the special case of the classical rod equation. Our second goal is to 
handle more general boundary conditions in order to encompass the case of  solutions not necessarily vanishing at infinity.
In particular, in the present paper we will be able to cover the case $f(u)=u^2$ and $g(u)=\kappa u +u^2$
corresponding to the Camassa--Holm equation with dispersion ($\kappa>0$),  a case that was not covered in~\cite{Bra13}.
Contrary to previously known blowup criteria, like those in~\cites{CamHolHym94, ACon00, ConEschActa, ConStra00, LiOlver00, TYZ14,Yin04,Zhou2004}, our criterion has the specific feature of being {\it purely local\/} in the space variable: indeed our blowup condition only involves the values of $u_0(x_0)$ and $u'_0(x_0)$ in a single  point $x_0$ of the real line. On the other hand, for earlier criteria, checking the blowup conditions involved the computation of global quantities (typically, the  $\|u_0\|_{H^1}$ norm, as in criterion \eqref{blow:TYZ} above) or other global conditions  like antisymmetry assumptions or sign conditions on the associate potential.

As we shall see, in order to establish such blowup result we will need to restrict the choice of the admissible functions~$f$ and~$g$.
On the other hand, when available, our criterion is applicable to a wider class of initial data.
We are now in the position of establishing our theorem.
Roughly speaking, under appropriate conditions on~$f$ and~$g$, we get the finite time 
blowup as soon as 
\begin{equation}
 \label{blo:cond}
 \exists \, x_0\in\R\quad\text{such that}\quad
 u_0'(x_0)<-\beta\bigl|u_0(x_0)-c\bigr|,
 \end{equation}
where $\beta$ and $c$ are two real constants depending on the shape of the functions~$f$ and $g$.

We obtain two slightly different versions  when $g$ is bounded from below,
or when~$g$ is bounded from above. Both cases turn out to be physically interesting.

\begin{theorem}
\label{th:blowup}
Let $u_0\in H^s(\R)$, with $s>3/2$. 
Let $ f,g\in C^\infty(\R)$ with $f''\ge\gamma>0$.
The maximal time~$T^*$ of the solution~$u$ to 
problem~\eqref{genrod} 
in $C([0,T^*),H^s(\R))\cap C^1([0,T^*),H^{s-1}(\R))$ must be finite, if at least one of the two
following conditions~\eqref{C1} or~\eqref{C2} is fulfilled:
\begin{enumerate}
 \item
 \label{C1}
\begin{itemize}
\item[-] $\exists\,c\in\R$ such that
$m=g(c)=\min_{\R} g$.
\item[-]
The map $\phi\colon\R\to\R$ given by
$\phi=\sqrt{\textstyle\frac{1}{\gamma}(g-m)}$ is $K$-Lipschitz with $0\le K\le 1$,
\item[-]
$\exists\, x_0\in\R\quad \text{such that}\qquad
u_0'(x_0)<-\frac{1}{2}\Bigl(\sqrt{1+8K^2}\,-1\Bigr) |u_0(x_0)-c|.$
\end{itemize}
\item
\label{C2}Or, otherwise,
\begin{itemize}
\item[-] $\exists\,c\in\R$ such that
$M=g(c)=\max_{\R} g$.
\item[-]
The map $\psi\colon\R\to\R$ given by
$\psi=\sqrt{\textstyle\frac{1}{\gamma}(M-g)}$ is $K$-Lipschitz with~\mbox{$0\le K\le \frac{1}{\sqrt8}$}
\item[-]
$\exists\, x_0\in\R\quad \text{such that}\qquad
u_0'(x_0)<-\frac{1}{2}\Bigl(1-\sqrt{1-8K^2}\Bigr) |u_0(x_0)-c|.$
\end{itemize}
\end{enumerate}
More precisely, the following upper bound estimate for~$T^*$ holds:
\begin{equation}
\label{estT}
 T^*\le \frac{4}{\gamma\sqrt{4u_0'(x_0)^2-\Bigl(\sqrt{1\pm8K^2}-1\Bigr)^2\bigl(u_0(x_0)-c\bigr)^2}\,\,},
\end{equation}
where in the term $\pm8K^2$ one has to take the positive sign under the conditions of Part~\eqref{C1}
and the negative sign under the conditions of~Part~\eqref{C2}.
\end{theorem}

We will also establish a variant of this theorem for solutions that are not in $H^s(\R)$. 
See Theorem~\ref{th:nonvan} below. This variant will apply to a large class of infinite energy solutions, as well as  to 
solutions with non-vanishing and possibly different asymptotics as $x\to+\infty$ and $x\to-\infty$.
Another possible variant is obtained assuming that $f$ is strictly concave, rather than strictly convex: in this case the blowup
condition on $u_0$ should be of the form: $u_0'(x_0)>\beta|u_0(x_0)-c|$.

\begin{remark}[Application to the Camassa--Holm equation]
The case $f(u)=\frac12 u^2$ and $g(u)=\kappa u +u^2$ corresponds to  the Camassa--Holm equation with dispersion~\eqref{CH}.
Situation~\eqref{C1} of Theorem~\ref{th:blowup} applies (with $c=-\kappa/2$, $\phi(u)=\sqrt{u^2+\kappa u+{\kappa^2}/{4}}$
and $K=1$).
We then immediately get the following corollary:

\begin{corollary}
\label{cor:CH}
Let $u_0\in H^s(\R)$, with $s>3/2$ be such that at some point~$x_0\in\R$ we have
\[
 u_0'(x_0)<-\bigl|u_0(x_0)+\textstyle\frac{\kappa}{2}\bigr|.
\]
Then the corresponding solution of the Camassa--Holm equation breaks down in finite time.
\end{corollary}

This corollary could be also obtained from the special case $\kappa=0$ (established in~\cite{Bra13}) 
via the change of unknown $v(t,x)=u(t,x-\frac{\kappa}{2}t)+\frac\kappa2$. Indeed, if~$u$ solves equation~\eqref{CH},
then~$v$ solves the Camassa--Holm equation without dispersion. However, one should check
that the proof of~\cite{Bra13} does indeed go through when applied to~$v$, which requires slight changes 
(the point is that the solution $v$ does not vanish as $x\to\infty$ as it was required in~\cite{Bra13}).
\end{remark}

\begin{remark}[Application to the classical rod equation]
In the case $f(u)=\frac{\gamma}{2}u^2$ and $g(u)=\frac{3-\gamma}{2}u^2$, corresponding to the classical rod equation,
the conditions of our theorem are satisfied if and only if~$1\le\gamma\le4$. Namely, situation~\eqref{C1} applies for~$1\le\gamma\le3$ and situation~\eqref{C2} applies for~$3\le\gamma\le4$.
Making explicit our blowup condition in this case, we see that solutions of the classical rod equation break down in finite time
as soon as, at some point $x_0\in\R$, we have 
\begin{equation}
\label{res:cmp}
u_0'(x_0)<- \frac{1}{2\sqrt\gamma}\biggl|\sqrt{12-3\gamma}-\sqrt\gamma\biggr|\,|u_0(x_0)| \qquad(1\le\gamma\le4).
\end{equation}
This conclusion allow us to recover the result in~\cite{Bra13}. 
Outside the range $1\le\gamma\le4$ 
it seems difficult to get purely ``local-in-space'' blowup criteria in the same spirit as in 
Theorem~\ref{th:blowup}. But non local-in-space blowup conditions involving the computation of the $\|u_0\|_{H^1}$ as in~\eqref{blow:TYZ} still apply outside the above range for the parameter~$\gamma$, see, {\it. e.g.\/}, 
\cite{ConStra00}, \cite{Guo-Zhou-SIAM}.
\end{remark}

Next section will be devoted to the proof of Theorem~\ref{th:blowup}.
The most immediate application of this theorem  is that one can recover earlier blowup criteria, like that in~\eqref{blow:TYZ}, by a simple application of Sobolev imbedding theorems.
In the last section we discuss some further consequences of of Theorem~\ref{th:blowup} by establishing three corollaries.
Such theorem and its variant, Theorem~\ref{th:nonvan}, imply that global solutions must satisfy quite stringent properties,
including a not too fast decay of $u(t,x)-c$  as $|x|\to\infty$, and  sign restrictions for $u(t,x)-c$.
Namely, we prove that the finite time blowup must occur if $u_0(x)=c+o(e^{-\beta|x|})$ as $|x|\to \infty$.
The finite time blowup must occur also if 
there exist $x_1<x_2$ such that $u(t,x_1)>c>u(t,x_2)$. Here, the constants 
constants $\beta>0$ and $c\in\R$ are as in Theorem~\ref{th:blowup}.

\section{The proof of Theorem~\ref{th:blowup}}

\begin{proof}[Proof of Theorem~\ref{th:blowup}]
 We can assume, without restriction, that $s\ge3$. Indeed, if $3/2<s<3$, and $u_0\in H^s(\R)$ satisfies a condition of the 
 form~\eqref{blo:cond}, then  we can  approximate $u_0$ with a sequence of data belonging to $H^3(\R)$ and satisfying 
 the same condition~\eqref{blo:cond}, and next use the well-posedness result recalled in Theorem~\ref{th:TYZ}.
 In what follows we then consider a solution $u\in C([0,T^*)\cap H^3)\cap C^1([0,T^*),H^{2})$. 

We start with an useful lemma:
\begin{lemma}
\label{lem:ineq}
Let ${\bf 1}_{\R^\pm}$ denote one of the two indicator functions~${\bf 1}_{\R^+}$ or ${\bf 1}_{\R^-}$.
\begin{enumerate}
 \item 
 If $f$ and $g$ satisfy the condition as in \mbox{Theorem~\ref{th:blowup}\,-\,\eqref{C1}}
 then the following estimate holds:
\begin{equation}
 \label{ineq}
 \begin{split}
  & (p\mathbf{1}_{\R^\pm})*\biggl(g(u)+\frac{f''(u)}{2}u_x^2\biggr)\ge\frac{\alpha}{2}\Bigl(g(u)-m\Bigr)+\frac m2
 \end{split}
 \end{equation}
with
\begin{equation}
\label{choice-a}
\alpha=\frac{1}{4K^2}\Bigl(\sqrt{1+8K^2}-1\Bigr).
\end{equation}
\item
If $f$ and $g$ satisfy the condition as in \mbox{Theorem~\ref{th:blowup}\,-\,\eqref{C2}},
then we have:
\begin{equation}
 \label{ineq2}
 \begin{split}
  & (p\mathbf{1}_{\R^\pm})*\biggl(g(u)+\frac{f''(u)}{2}u_x^2\biggr)\ge\frac{\alpha}{2}\Bigl(g(u)-M\Bigr)+\frac M2
 \end{split}
 \end{equation}
with
\begin{equation}
\label{choice-a2}
\alpha=\frac{1}{4K^2}\Bigl(1-\sqrt{1-8K^2}\Bigr).
\end{equation}
\end{enumerate}
In the case $g=m=M$ be a constant function 
(this corresponds to $K=0$), the right-hand side of the above convolution estimates reads $(p\mathbf{1}_{\R^\pm})*\bigl(g+\frac{f''(u)}{2}u_x^2\bigr)\ge g/2$.
\end{lemma}

\begin{proof}
Let us consider the case of the indicator function~${\bf 1}_{\R^+}$.

Under the conditions of \mbox{Theorem~\ref{th:blowup}\,-\,\eqref{C1}} for~$f$ and $g$,
we see that  $g-m\ge0$ and that the zeros of $g-m$ are of order at least two. Then $\phi=\sqrt{\frac{1}{\gamma}(g-m)}$
is differentiable, and by the Lipschitz condition $|\phi'|\le K$.
This in turn implies
\[
(g')^2\le 4\gamma K^2(g-m).
\]
Consider the quadratic polynomial in~$\lambda$,
\[
P(\lambda)=\frac{\gamma}{2}\lambda^2-\alpha g'(u)\lambda+b\bigl(g(u)-m\bigr),
\]
where $\alpha$ and $b$ are constants to be determined below.
We have  $P\ge0$  on the real line if and only if $\alpha^2(g'(u))^2\le 2\gamma b \bigl(g(u)-m\bigr)$.
We then choose 
\[
b=2\alpha^2K^2,
\]
thus ensuring that indeed $P\ge0$.
Recalling  $f''(u)\ge \gamma$ and using the fact that $P(u_x(\xi))\ge0$ for all real $\xi$, we now deduce that
\begin{equation*}
 \begin{split}
  (p\mathbf{1}_{\R^+})*\biggl(b\bigl(g(u)-m\bigr)+\frac{f''(u)}{2}u_x^2\biggr)(x)
  &=\frac{e^{-x}}{2}\int_{-\infty}^x e^\xi\biggl(b\bigl(g(u)-m\bigr)+\frac{f''(u)}{2}u_x^2\biggr)(\xi)\,d\xi\\
  &\ge \frac{e^{-x}}{2}\int_{-\infty}^x e^\xi\bigl( \alpha g'(u)u_x\bigr)(\xi)\,d\xi\\
  &= \alpha \frac{e^{-x}}{2}\int_{-\infty}^x e^\xi\frac{d}{d\xi}\Bigl[ g(u)-m\Bigr](\xi)\,d\xi\\
  &=\frac{\alpha}{2}\bigl(g(u)-m\bigr)(x)-\alpha (p\mathbf{1}_{\R^+})*\bigl(g(u)-m\bigr)(x).
  \end{split}
\end{equation*}

Hence,
\begin{equation}
(p\mathbf{1}_{\R^+})*\biggl((b+\alpha)\bigl(g(u)-m\bigr)+\frac{f''(u)}{2}u_x^2\biggr)(x)\ge \frac\alpha2 \bigl(g(u)-m\bigr)(x).
\end{equation}
In order to get the first  estimate of the lemma we need
\[
b+\alpha=1.
\]
This boils down to the equation $2K^2\alpha^2+\alpha-1=0$.
Taking the largest real root we obtain expression~\eqref{choice-a} for~$\alpha$.
Now observing that $(p{\bf 1}_{\R^+})*m=m\int_\R p{\bf 1}_{\R^+}=m/2$, 
leads to estimate~\eqref{ineq}.

The reasoning when~$f$ and~$g$ satisfy the conditions of \mbox{Theorem~\ref{th:blowup}\,-\,\eqref{C2}} is entirely analogous:
one has  $(g')^2\le 4\gamma K^2(M-g)$.
We can thus consider the same polynomial as before, but choosing now $b=-2\alpha^2 K^2$. Reproducing the same computation as
above (with $M$ instead of~$m$) we find as before the relation $\alpha+b=1$. 
Eliminating~$b$ gives the equation $2K^2\alpha^2-\alpha+1=0$. We thus need $0\le K\le 1/\sqrt 8$ for this equation to admits real roots.
Choosing now the smallest root gives~\eqref{choice-a2} and estimate~\eqref{ineq2}.

The convolution estimate involving~$p{\bf 1}_{\R^-}$ can be proved in a similar way.
Indeed, under, {\it e.g.\/} the assumptions  of \mbox{Theorem~\ref{th:blowup}\,-\,\eqref{C1}}
we have, for any $\lambda\in\R$,
\[
\widetilde P(\lambda)=
\frac{\gamma}{2}\lambda^2+\alpha g'(u)\lambda+b(g(u)-m)\ge0.
\]
Thus,
\begin{equation*}
 \begin{split}
  (p\mathbf{1}_{\R^-})*\biggl(b\bigl(g(u)-m\bigr)+\frac{f''(u)}{2}u_x^2\biggr)(x)
  &=\frac{e^{x}}{2}\int_{x}^\infty e^{-\xi}\biggl(b\bigl(g(u)-m\bigr)+\frac{f''(u)}{2}u_x^2\biggr)(\xi)\,d\xi\\
  &\ge \frac{e^{x}}{2}\int_{x}^{+\infty} e^{-\xi}\bigl( -\alpha g'(u)u_x\bigr)(\xi)\,d\xi\\
  &=\frac{\alpha}{2}\bigl(g(u)-m\bigr)(x)-\alpha (p\mathbf{1}_{\R^-})*\bigl(g(u)-m\bigr)(x),
  \end{split}
\end{equation*}
and we can proceed as in the previous case.
\end{proof}

Let us go back to the proof of Theorem~\ref{th:blowup}.
Taking the space derivative in equation~\eqref{genrod}, and recalling that 
$(1-\partial_x^2)p$ equals the Dirac mass at the origin, we get
\begin{equation}
\label{genrodx}
u_{tx}+f'(u)u_{xx}=-\frac{f''(u)}{2}u_x^2+g(u)-p*\Bigl[g(u)+\textstyle\frac{f''(u)}{2}u_x^2\Bigr].
\end{equation}

We now consider the flow map, defined by
\begin{equation}
\label{flow-q}
\begin{cases}
q_t(t,x)=f'(u(t,q(t,x)), & t>0,\quad x\in\R,\\
q(0,x)=x & x\in\R,
\end{cases}
\end{equation}
where $u$ is the solution of the problem~\eqref{genrod} given by Theorem~\ref{th:TYZ}.
Notice that the assumptions made on~$f$ and~$u$ imply that $q\in C^1([0,T^*)\times\R,\R)$ 
is well defined on the whole time interval $[0,T^*)$.

We now proceed putting the conditions  of \mbox{Theorem~\ref{th:blowup}\,\eqref{C1}}.
From~\eqref{genrodx}, the uniform convexity condition $f''\ge\gamma>0$, and summing up the two convolution estimates for ${\bf 1}_{\R^+}$ and ${\bf 1}_{\R^-}$
in~\eqref{ineq}, we get
\begin{equation}
\label{begest}
\begin{split}
\frac{d}{dt}\bigl[u_x(t,q(t,x))]
&=\Bigl[u_{tx}+f'(u)u_{xx}\Bigr](t,q(t,x))\\
&=-\frac{f''(u)}{2}u_x^2+g(u)-p*\Bigl(g(u)+\frac{f''(u)}{2}u_x^2\Bigr)\\
&\le \Bigl[-\frac{\gamma}{2}u_x^2+(1-\alpha)\bigl(g(u)-m\bigr)\Bigr](t,q(t,x)).
\end{split}
\end{equation}
By the definition of~$\alpha$~\eqref{choice-a} we see that $0<\alpha\le1$.
We can express $K$ in terms of $\alpha$ as
\begin{equation}
\label{relaK}
2K^2=\frac{1-\alpha}{\alpha^2}.
\end{equation}
The Lipschitz condition on 
\[\phi=\sqrt{\textstyle\frac{1}{\gamma}(g-m)}\]
provides the estimate 
$\phi(u)\le K|u-c|$.
We thus obtain
\begin{equation}
\label{calcc}
\begin{split}
\frac{d}{dt}\bigl[u_x(t,q(t,x))]
&\le -\frac{\gamma}{2}u_x^2+(1-\alpha) \gamma\phi(u)^2\\
&\le -\frac{\gamma}{2}u_x^2+(1-\alpha) K^2\gamma(u-c)^2\\
&=\frac{\gamma}{2}\biggl(\frac{(1-\alpha)^2}{\alpha^2}(u-c)^2-u_x^2\biggr).
\end{split}
\end{equation}
Let us set
\begin{equation}
 \label{beta-exp}
\beta:=\frac{1-\alpha}{\alpha}=2K^2\alpha=\textstyle\frac{1}{2}\Bigl(\sqrt{1+8K^2}-1\Bigr).
\end{equation}
Next introduce
\begin{equation}
\label{def:A}
A(t,x)=\bigl(\beta (u-c)-u_x\bigr)(t,q(t,x))
\end{equation}
and
\begin{equation}
\label{def:B}
B(t,x)=\bigl(\beta (u-c)+u_x\bigr)(t,q(t,x))
\end{equation}

We then obtain from~\eqref{calcc}
\begin{equation}
\label{lat}
\frac{d}{dt}\bigl[u_x(t,q(t,x))]
\le \frac{\gamma}{2}(AB)(t,x).
\end{equation}

On the other hand, the kernel~$p$ satisfies the identity (both in the distributional and a.e. pointwise sense)
\[
 \partial_xp=p{\bf 1}_{\R^-}-p{\bf 1}_{\R^+}.
\]
Then we get, recalling the inequality $f''\ge\gamma$,
\begin{equation}
\label{att}
\begin{split}
A_t(t,x)
&=\beta(u_t+f'(u)u_x)-(u_{tx}+f'(u)u_{xx})\\
&=\frac{f''(u)}{2}u_x^2-g(u)+(p-\beta\partial_xp)*\biggl(g(u)+\frac{f''(u)}{2}u_x^2\biggr)\\
&\ge \frac{\gamma}{2}u_x^2-g(u)+(1+\beta)p\mathbf{1}_{\R^+}*\biggl(g(u)+\frac{f''(u)}{2}\biggr)
  +(1-\beta)p\mathbf{1}_{\R^-}*\biggl(g(u)+\frac{f''(u)}{2}\biggr).
\end{split}
\end{equation}
We now would like to apply the convolution estimates~\eqref{ineq}.
This can be done, provided we have
$-1\le \beta\le1$. Such additional condition is equivalent to $\alpha\ge1/2$ and this last condition
is ensured by the restriction $0\le K\le1$ made in the assumptions of~\mbox{Theorem~\ref{th:blowup}\,-\eqref{C1}}.
Now applying estimates~\eqref{ineq} and reproducing the same calculations as  in~\eqref{calcc} gives 
\begin{equation}
\label{in-A}
\begin{split}
A_t(t,x)
&\ge \frac{\gamma}{2}u_x^2+(\alpha-1)\bigl(g(u)-m\bigr)\\
&= \frac{\gamma}{2}u_x^2-(1-\alpha)\gamma \phi(u)^2\\
&\ge\frac{\gamma}{2}\bigl(u_x^2-\beta^2(u-c)^2\bigr)\\
&=-\frac{\gamma}{2}(AB)(t,x).
\end{split}
\end{equation}

Similar computations yield the estimate
\begin{equation}
\label{in-B}
\begin{split}
B_t(t,x)
&\le -\frac{\gamma}{2}u_x^2+(1-\alpha)\bigl(g(u)-m\bigr)\\
&\le -\frac{\gamma}{2}u_x^2+(1-\alpha)\gamma\phi(u)^2 \\
&\le\frac{\gamma}{2}\bigl(\beta^2(u-c)^2-u_x^2\bigr)\\
&=\frac{\gamma}{2}(AB)(t,x).
\end{split}
\end{equation}

By our assumption on the initial datum made in Part~\eqref{C1} of Theorem~\ref{th:blowup},
\[
u_0'(x_0)<-\frac{1}{2}\Bigl(\sqrt{1+8K^2}\,-1\Bigr) |u_0(x_0)-c|.
\]
According to the definition of~$\beta$~\eqref{beta-exp}, this can be re-expressed as
\begin{equation*}
u_0'(x_0)<-\beta\bigl|u_0(x_0)-c\bigr|,
\end{equation*}
or, equivalently, as
\[A(0,x_0)>0\qquad\text{and}\qquad B(0,x_0)<0.\]
Let 
\[
\tau=\sup\bigl\{t\in[0,T^*)\colon A(\cdot,x_0)>0 \text{ and } B(\cdot,x_0)<0 \text{ on } [0,t]\bigr\}.
\]
By continuity, $\tau>0$. If $\tau<T^*$, then at least one of the inequalities $A(\tau,x_0)\le0$ and $B(\tau,x_0)\ge0$ must hold true.
This contradicts the fact that on the interval
$[0,\tau[$, we have $(AB)(\cdot,x_0)<0$, hence $A(\tau,x_0)\ge A(0,x_0)>0$ by~\eqref{in-A} and $B(\tau,x_0)\le B(0,x_0)<0$
by~\eqref{in-B}.
Thus $\tau=T^*$. 
Summarizing, we can say that during the whole existence time $[0,T^*)$:
\begin{equation}
\label{syst}
\begin{cases}
\text{$A(\cdot,x_0)$ is positive and increasing},\\
\text{$B(\cdot,x_0)$ is negative decreasing},\\
\text{$AB(\cdot,x_0)$ negative and decreasing}.
\end{cases}
\end{equation}

We are now in the position of proving that $T^*<\infty$:
let us consider
\[ h(t)=\sqrt{-(AB)(t,x_0)}.\]
Computing the time derivative of~$h$, next applying the differential inequalities~\eqref{in-A} and~\eqref{in-B} and
the geometric-arithmetic mean inequality $\frac{A-B}2(t,x_0)\ge h(t)$, we get
\[
 \begin{split}
 \frac{\dd h}{\dd t}(t)&=-\frac{A_tB+AB_t}{2\sqrt{-AB}}(t,x_0) \\
 &\ge \frac{\gamma\bigl(-AB\bigr)(A-B)}{4\sqrt{-AB}}(t,x_0)\\
 &\ge \frac{\gamma}{2}h^2(t).
 \end{split}
\]
But $h(0)=\sqrt{-AB(0,x_0)}>0$. Hence the solution blows up in finite time and $T^*< \frac{2}{\gamma h(0)}$.
Recalling the definitions of~$A$, $B$ we thus get the estimate for~$T^*$
\begin{equation*}  
 \label{est-T*}
 T^*\le \frac{2}{\gamma\sqrt{u_0'(x_0)^2-\beta^2\bigl(u_0(x_0)^2-c\bigr)^2}}
\end{equation*}
that agrees with that given in~\eqref{estT}

\medskip
The necessary changes to deal with the conditions of Part~\eqref{C2} of the theorem are slight.
First of all,  the relation between~$K$ and $\alpha$ is now
\begin{equation*}
\label{relaK2}
2K^2=\frac{\alpha-1}{\alpha^2}.
\end{equation*}
instead of~\eqref{relaK}.
On the other hand, owing to~\eqref{choice-a2}, we now have $\alpha\ge1$.
Then we can replace estimates~\eqref{begest}-\eqref{calcc} with
\[
 \begin{split}
\frac{d}{dt}\bigl[u_x(t,q(t,x))]
&=\Bigl[u_{tx}+f'(u)u_{xx}\Bigr](t,q(t,x))\\
&\le \Bigl[-\frac{\gamma}{2}u_x^2+(\alpha-1)\bigl(M-g(u)\bigr)\Bigr](t,q(t,x))\\
&\le -\frac{\gamma}{2}u_x^2+(\alpha-1) \gamma\psi(u)^2\\
&\le\frac{\gamma}{2}\biggl(\frac{(\alpha-1)^2}{\alpha^2}(u-c)^2-u_x^2\biggr).
\end{split}
\]
As the coefficient~$\beta$ now given by $\beta:=2K^2\alpha=\textstyle\frac{1}{2}\Bigl(1-\sqrt{1-8K^2}\Bigr)$, the required condition $-1\le\beta\le1$ does not bring any additional restriction on the Lipschitz constant~$K$.
The last part of the proof proceeds in the same manner. 
\end{proof}

\begin{remark}
 In some cases, the equality holds in~\eqref{ineq} and~\eqref{ineq2} for some specific choices of the function $u$.
 This happens ({\it e.g.\/} in the case~\eqref{C1}) when the functions $g$ and $f$ are such that
 \begin{equation}
  \label{D=0}
  2\alpha^2(\phi')^2=(1-\alpha)f'' \quad\text{in all points where $\phi=\sqrt{\textstyle\frac1\gamma(g-m)}$ does not vanish}.
  \end{equation}
Indeed, such condition tells us that the discriminants of $P$ and $\widetilde P$ vanish.
In this case, take $u$ such that, at some point $x_0$,
\begin{equation}
 \label{solit}
 \begin{cases}
  u_x(t,\xi)=\frac{\alpha g'(u(t,\xi))}{f''(u(t,\xi))}, &\text{if $\xi<x_0$},\\
  u_x(t,\xi)=-\frac{\alpha g'(u(t,\xi))}{f''(u(t,\xi))}, &\text{if $\xi>x_0$}.
 \end{cases}
 \end{equation}
Then we have $P(u_x(t,x_0))=\widetilde P(u_x(t,x_0))=0$.
This in turn implies that the equality holds at the point $x_0$ in~\eqref{ineq}.
In the particular case of  the Camassa--Holm equation, $f(u)=u^2/2$, $g(u)=u^2$ and $\alpha=1/2$.
Thus condition~\eqref{D=0} is satisfied. Next imposing the continuity of~$u$, the system~\eqref{solit} boils down to the differential equations defining the peaked solitons $u_c(x,t)=ce^{|x-ct|}$.
\end{remark}

\section{Futher consequences and conclusions }
We established the finite time blowup for solutions of the generalized rod equation
\begin{equation*}
 u_t-u_{txx}+[f(u)]_x-[f(u)]_{xxx}+\Bigl[g(u)+\frac{f''(u)}{2}u_x^2\Bigr]_x=0,
\end{equation*}
under appropriate conditions on the functions~$f$ and~$g$, provided the initial datum~$u_0$  ($s>3/2$), satisfies
\begin{equation*}
 \exists \, x_0\in\R\quad\text{such that}\quad
 u_0'(x_0)<-\beta\bigl|u_0(x_0)-c\bigr|,
 \end{equation*}
where $\beta$ and $c$ are two real constants depending on the shape of the functions~~$f$ and $g$.

In stating Theorem~\ref{th:blowup}, we considered for simplicity 
solutions $C([0,T^*),H^s(\R))\cap C^1([0,T^*),H^{s-1}(\R))$. These are known to uniquely exist
provided the initial datum is such that $u_0\in H^s(\R)$, with $s>3/2$.
In particular, such solutions are of finite energy and vanish as $|x|\to\infty$.

However, a closer look to the proof of Theorem~\ref{th:blowup} reveals that our arguments go through also in the case
of {\it infinite energy solutions\/}, possibly {\it non-vanishing\/} at infinity. 
Of course, the problem arises of finding a suitable functional setting in order to get the local existence and the uniqueness
of this type of solutions. This problem has been successfully addressed in~\cite{GHR2012}, at least in the Camassa--Holm case and for solutions admitting possibly distinct limits as  $x\to\pm \infty$.

Rather than restating Theorem~\ref{th:blowup} relying on well-posedness results more general than those in the $H^s$-setting,
let us assume that we are given {\it a priori\/} a solution $u\in C^1([0,T^*),C^2(\R))$ of the generalized rod equation, 
written in its non-local form~\eqref{genrod}, for some $0<T^*\le +\infty$.
We need to put also two {\it a priori\/} growth conditions on~$u$:
\begin{itemize}
\item[(i)] We first assume  that~$u$ is such that  the integrals in Lemma~\ref{lem:ineq} converge (in particular the convolution  term
in~\eqref{genrod} makes sense): this is a very mild  condition, that would allow us  to take into account also unbounded solutions
({\it e.g.\/} solutions with polynomial growth, when $g$ and $f$ are polynomials).
\item[(ii)]
We  assume that the flow map in~\eqref{flow-q} is well defined on the whole time interval~$[0,T^*)$:
this leads us to restrict ourselves to solutions~$u$ such that  $|f'\circ u|$ is bounded, uniformly on compact time intervals,
by an affine function of the~$x$ variable.
\end{itemize}
The above conditions, in particular, encompass the case of smooth  solutions $u$ to~\eqref{genrod}, such that
both $u(t,\cdot)$ and $u_x(t,\cdot)$ are bounded on~$\R$ (uniformly with respect to~$t$ on compact time intervals).

Hence, we obtain the following variant of~Theorem~\ref{th:blowup}:

\begin{theorem}
\label{th:nonvan}
Let $ f,g\in C^\infty(\R)$ with $f''\ge\gamma>0$.
Let $u\in C([0,T^*),W^{1,\infty}(\R))\cap C^1([0,T^*),C^2(\R))$, with $0<T^*\le\infty$,
be a solution of the Cauchy problem for the generalized rod equation~\eqref{genrod}. 
Assume that  at least one of the conditions~\eqref{C1} 
or~\eqref{C2} of Theorem~\ref{th:blowup} is fulfilled. Then $T^*<\infty$ and~$T^*$ is bounded from above by~\eqref{estT}.

In particular, if $g$ is constant and  $u_0$ is non-constant, then $u$ blows up in finite time.
\end{theorem}

The very last statement of Theorem~\ref{th:nonvan} generalizes the known fact, \cite{ConStra00}, that for the classical rod equation in the limit case $\gamma=3$ (that corresponds to $g\equiv0$), any nonzero initial data gives rise to a solution that blows up.

The above discussion applies in particular to periodic solutions: hence, the statement of Theorem~\ref{th:blowup} 
remains valid for solutions~$u\in C([0,T^*),H^s(\S))\cap C^1([0,T^*),H^{s-1}(\S))$, with $s>3/2$, 
where $\S$ denotes the one-dimensional torus.
It should be pointed, however, that the estimates of Lemma~\ref{lem:ineq} (that are optimal for $u\in H^s(\R)$, at least for
a few specific choices of $f$ and~$g$) are no longer optimal when~$u\in H^s(\S)$. Improving such estimates would require the application
of variational methods. Therefore, we expect that in the case of the torus, the restriction on the Lipschitz constant~$K$ appearing
in Theorem~\ref{th:blowup} could be relaxed and the estimate on the coefficient~$\beta$ appearing in~\eqref{blo:cond} improved.
See~\cite{BraCor14} for results in this direction in the case of the classical period rod equation.

We finish this paper establishing three simple corollaries of our blowup results.
The first corollary establishes a relation between the behavior at the spatial infinity and the blow up.

\begin{corollary}
 \label{cor:decay}
 Let $ f,g\in C^\infty(\R)$ with $f''\ge\gamma>0$ be such that  at least one of the two following conditions is satisfied
 (the maps $\phi$ and $\psi$ are as in Theorem~\ref{th:blowup}):
 \begin{enumerate}
 \item
 \label{C11}
$\min_{\R} g=g(c)$ and $\phi$ is $K$-Lipschitz with~\mbox{$0\le K\le 1$}, or otherwise
\item
\label{C22}
$\max_{\R} g=g(c)$ and $\psi$ is $K$-Lipschitz with~\mbox{$0\le K\le \frac{1}{\sqrt8}$}
\end{enumerate}

If~$u\in C([0,T^*),W^{1,\infty}(\R))$ is a smooth solution of the
generalized rod equation~\eqref{gr-pde} arising from an initial datum~$u_0\not\equiv c$ such that
\begin{equation}
 \label{dec:cond}
\liminf_{x\to+\infty}e^{\beta x}\bigl(u_0(x)-c\bigr)\le 0
  \quad\text{and}\quad
  \limsup_{x\to-\infty}e^{-\beta x}\bigl(u_0(x)-c\bigr)\ge 0,
\end{equation}
where $\beta=\frac{1}{2}\bigl(\sqrt{1+8K^2}\,-1\bigr)$ (in the case~\eqref{C11}), or  
$\beta=\frac{1}{2}\bigl(1-\sqrt{1-8K^2}\bigr)$ (in the case~\eqref{C22}),
then $u$ must blow up in finite time.

In particular, a blow up occurs if $u_0\not\equiv c$ is such that 
$u_0(x)=c+o\bigl(e^{-\beta|x|}\bigr)$
for $|x|\to\infty$.
\end{corollary}

A result in the same direction as in Corollary~\ref{cor:decay} appeared in~\cite{Bra11} in the case of the Camassa--Holm equation (see also~\cite{HMPZ07} for earlier results in the same spirit),
where well-posedness issues in weighted spaces were thoroughly discussed.
However, the proof given in~\cite{Bra11} does not go through for equation~\eqref{genrod}, as  it deeply relies on Mc~Kean's necessary and sufficient condition for the global existence of solutions of the Camassa--Holm
equations.

The second corollary establishes a relation between blow up and sign changes for $u_0-c$.

\begin{corollary}
 \label{cor:u=c}
 Let $ f,g\in C^\infty(\R)$ as in Corollary~\ref{cor:decay} and let $u\in C([0,\infty),W^{1,\infty}(\R))$ be a global smooth solution of the rod equation ~\eqref{gr-pde}. Then, for all $t\ge0$, 
\begin{itemize}
 \item[i)] Either $u(t,x)>c$ for all $x\in\R$,
  \item[ii)] or $u(t,x)< c$ for all $x\in\R$,
  \item[iii)] or $\exists \,x_t\in\R$ such that $u(t,\cdot)\le c$ in $(-\infty,x_t]$ and $u(t,x)\ge c$ in $[x_t,+\infty)$.
 In this case, if $x\mapsto u(t,x)$ is equal to~$c$ at two distinct points of the real line, then $x\mapsto u(t,x)$
 must be constant~$=c$ in the whole interval between them.
\end{itemize}
\end{corollary}

Our last corollary is a unique continuation result valid for \emph{periodic} solutions.
Let $\S=\R/\Z$ denote the one-dimensional torus. The local-in-time well posedness result of periodic solutions in
$C([0,T),H^s(\S))\cap C^1([0,T),H^{s-1}(\S))$ for $s>3/2$, together with the mean  $\int_\S u(t)\dd x=\int_\S u_0\dd x$
and the energy $\int_{\S} (u^2+u_x^2)(t)\dd x=\int_{\S}(u_0^2+(u_0)_x^2)\dd x$ conservation laws,
can be established following the steps of~\cite{TYZ14}.

\begin{corollary}
 \label{cor:per}
 Under the conditions of Corollary~\ref{cor:decay} for~$f$ and~$g$, the identically constant solution~$u\equiv c$ (where $c=\mbox{\text{\rm arg\,min\,$g$}}$,  or $c=\mbox{\text{\rm arg\,max\,$g$}}$) is the only  global, smooth and 
 spatially periodic solution  of the generalized rod equation~\eqref{gr-pde} with time-independent energy,  such that 
 $u_0(x_0)=c$ for some some $x_0\in\R$.
\end{corollary}

\begin{proof}[Proof of Corollary~\ref{cor:decay},~\ref{cor:u=c} and~\ref{cor:per}]
If $u$ is a global smooth solution in $C([0,\infty),W^{1,\infty}(\R))$ of the
generalized rod equation~\eqref{gr-pde}  then, by Theorem~\ref{th:nonvan}, for all $t\ge0$ and for all $x\in\R$ we must have
$u_x(t,x)\ge \beta|u(t,x)-c|$, with $\beta=\frac{1}{2}\bigl(\sqrt{1+8K^2}\,-1\bigr)$ (in the case~\eqref{C11}), or  
$\beta=\frac{1}{2}\bigl(1-\sqrt{1-8K^2}\bigr)$ (in the case~\eqref{C22}).

If $[a,b]$ is an interval where $u(t,x)\ge c$ for all $x\in[a,b]$, then we have 
\begin{align*}
 -\int_a^b\beta e^{\beta x}\bigl( u(t,x)-c\bigr)\dd x 
 &\le \int_a^be^{\beta x}u_x(t,x)\dd x\\
 &=e^{\beta b}\bigl(u(t,b)-c\bigr)-e^{\beta a}\bigl(u(t,a)-c\bigr)-\beta\int_a^be^{\beta x}(\bigl(u(t,x)-c\bigr)\dd x.
\end{align*}
This implies that
$e^{\beta a}\bigl(u(t,a)-c\bigr)\le e^{\beta b}\bigr(u(t,b)-c\bigr)$. 
 
If $u(t,\cdot)\le c$ on $[a,b]$, then applying again Theorem~\ref{th:nonvan} we have $u_x(t,x)\ge \beta_\gamma \bigl(u(t,x)-c\bigr)$.
This in turn implies
\begin{align*}
 \int_a^b \beta e^{-\beta x}  \bigl(u(t,x)-c\bigr)\dd x 
 &\le \int_a^b e^{-\beta x}u_x(t,x)\dd x\\
 &=e^{-\beta b}\bigl(u(t,b)-c\bigr)-e^{-\beta a}\bigl(u(t,a)-c\bigr)
 +\beta\int_a^be^{-\beta x}\bigl(u(t,x)-c\bigr)\dd x.
\end{align*}
We then conclude that
$e^{-\beta a}\bigl(u(t,a)-c\bigr)\le e^{-\beta b}\bigl(u(t,b)-c\bigr)$.

Summarizing, we proved that $x\mapsto e^{\beta x} \bigl(u(x,t)-c\bigr)$ is monotone increasing in any interval where
$u(\cdot, t)\ge c$ and $x\mapsto e^{-\beta x} \bigl(u(x,t)-c\bigr)$ is monotone increasing in any interval 
where $u(\cdot,t)\le c$.

Corollary~\ref{cor:decay},~\ref{cor:u=c} and~\ref{cor:per} now easily follow from these monotonicity properties.
Indeed,  if $u_0(x_0)>c$ at some point $x_0\in\R$, then for all $x>x_0$, we have $u_0(x)-c>e^{-\beta(x-x_0)}(u_0(x_0)-c)$,
hence  $\liminf_{x\to+\infty}e^{\beta x}\bigl(u_0(x)-c\bigr)> 0$.  
In the same way, we see that if $u_0(x_0)<c$, at some point $x_0$, then 
$\limsup_{x\to-\infty}e^{-\beta x}\bigl(u_0(x)-c\bigr)< 0$.
The claim of Corollary~\ref{cor:decay} immediately follows.
On the other hand, if we exclude the two situations~{i)} and~{ii)} in Corollary~\ref{cor:u=c}, 
then $u(t,x_t)=c$ at some point $x_t\in \R$; in the interval $(-\infty, x_t]$ we must have $u(t,x)\le c$:
otherwise, if $u-c$ is strictly positive at some point $x'<x_t$, then 
we would get,  for all $x\ge x'$, $u(t,x)-c>e^{-\beta(x-x')}(u(t,x')-c)$
and in particular $u(t,x_t)$ could not be equal to~$c$.
In the same way, we see that we must have $u(t,x)\ge c$
in the interval $[x_t,\infty)$. Such argument proves also that as soon $u(t,\cdot)$ vanishes at two points $x_t<y_t$,
then $u(t,\cdot)$ must vanish in the whole interval $[x_t,y_t]$.
The claim of Corollary~\ref{cor:u=c} follows.
In the periodic case, if $u$ is a global smooth solution such that $u_0(x_0)=c$ at some point~$x_0$, then by Corollary~\ref{cor:u=c}  $u_0\equiv c$. Inside the class of energy-preserving solutions, this in turn implies that $u\equiv c$.
Corollary~\ref{cor:per} is thus established.
\end{proof}

In the special case of the Camassa--Holm equation without dispertion ($\kappa=0$), Corollary~\ref{cor:u=c}
ensures that a sufficient condition for the breakdown is that the
initial data satisfy the sign condition \emph{$\exists\, x_1<x_2$ such that $u_0(x_1)>0>u_0(x_2)$}. 
This should be compared with the blowup condition of McKean's theorem, that reads:
  \emph{$\exists\, x_1<x_2$ such that $y_0(x_1)>0>y_0(x_2)$}, where $y_0=u_0-(u_0)_{xx}$ is the associated potential. 
 McKean's theorem is more precise, as it provides a {\it necessary and sufficient\/} condition  for the wave-breaking of the dispersionless Camassa--Holm equation, see~\cite{McKean04}.
However, as already observed, McKean's arguments deeply rely on the persistence properties of the sign of the potential $y$ during the evolution --- a byproduct of the bi-hamiltonian structure --- that are no longer valid for the general equation~\eqref{genrod}.


\end{document}